\documentclass{article}
\usepackage[utf8]{inputenc}
\usepackage{mathtools, amssymb, tikz-cd, adjustbox, verbatim, titlesec, setspace, MnSymbol, amsthm, indentfirst, graphicx, float, mathrsfs, tikz-network, fancyvrb}
\usepackage[none]{hyphenat}
\DeclarePairedDelimiter{\brac}{\{}{\}}

\usepackage[shortlabels]{enumitem}
\newtheorem{theorem}{Theorem}[section]
\newtheorem{corollary}[theorem]{Corollary}
\newtheorem{lemma}[theorem]{Lemma}
\newtheorem{proposition}[theorem]{Proposition}
\newtheorem{definition}[theorem]{Definition}
\renewenvironment{proof}{{\textit{Proof.}}}{\qed}
\usetikzlibrary{positioning}

\title{On the Assignment Graphs of Oriented Graphs}
\author{Jared Glassband, Garrison Koch,\\ Sophia Lebiere, Xufei Liu, and Evan Sabini}
\date{\today}

\begin{document}

\maketitle

\begin{abstract}	
In this paper, we extend the ideas of graph pebbling to oriented graphs and find a classification for all graphs with fully traversable pebbling assignments that are isomorphic to their assignment graph. We then give some cases in which a graph with a non-fully traversable pebbling assignment is isomorphic to its assignment graph.
\end{abstract}

\section{Introduction}

In this paper we require that $G$ be a finite, simple, oriented graph. An oriented graph being a directed graph in which no edge is bidirectional. To distinguish from the degree of a vertex in a standard graph let the \textbf{valence} of a vertex in an oriented graph be the number of edges that start at that vertex. Now we may again consider $(S_G)$ a \textit{pebbling assignment} on a graph $G$ some distribution of pebbles on the vertices of $G$. In the oriented case, a \textit{pebbling move} consists of removing two pebbles from vertex $v$ and simultaneously adding one pebble to a vertex $w$ where $(v,w)$ is an oriented edge of $G$. So we are restricting the possible pebbling moves of the standard graph pebbling by requiring that all pebbling moves must obey the orientation of the graph. We again consider the \textit{assignment graph} $[S_G]$ which is naturally directed as it is a Hasse diagram of the pebbling assignments on a graph after successive pebbling moves. We will be considering $G$ with assignment $(S_G)$ such that $G\cong [S_G]$ as directed graphs. Except for section 8, all graph isomorphism are isomorphisms of directed graphs. It follows from the standard assignment graphs that $[S_G]$ is a  bipartite, rooted graph. As before, a pebbling assignment $(S_G)$ is \textbf{fully traversable} if every edge of $G$ is traversed at some stage of the pebbling process. We will specify that a fully traversable graph must have an edge. By the paper by M. Lind et al., we have the following proposition.

\begin{proposition}
If $(S_G)$ is fully traversable and $G\cong [S_G]$ then every edge is traversed exactly once.
\end{proposition}

In the standard case, a \textit{movable} vertex was any connected vertex that had at least two pebbles on it. However, in the oriented case, we require that a movable vertex not only be connected with at least two pebbles on it but have non-zero valence and thus a possible pebbling move. Furthermore an $n$-movable vertex is a movable vertex with valence of at least $n$.

\begin{corollary}
Let $G$ be a graph with more than one vertex. If $(S_G)$ is fully traversable and $G\cong [S_G]$ then no vertex has more than three pebbles on it.
\end{corollary}

\begin{proof}
Suppose by way of contradiction $v\in V(G)$ has more than three pebbles on it. Then since $(S_G)$ is fully traversable, $v$ is movable. Since it has more than three pebbles on it, it can traverse any outgoing edge twice. Thus, by the previous proposition this is a contradiction.
\end{proof}

\begin{corollary}
If $(S_G)$ is not fully traversable and $G\cong [S_G]$ then some edge of $G$ is traversed more than once. 
\end{corollary}

\begin{proof}
If each edge were traversed at most once then $|E([S_G])|<|E(G)|$ since $(S_G)$ is not fully traversable. This is clearly a contradiction since $G\cong [S_G]$. 
\end{proof}

\section{Downward 4-Cycle}

The following graph is what we will call a \textbf{downward 4-cycle}. 

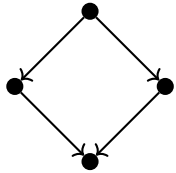
\begin{figure}[h]
\centering

\begin{tikzpicture}[thick, ->]

\tikzstyle{vertex}=[circle,fill=black,inner sep=2pt]

\node[draw, vertex] (1) at (0,0) {};
\node[draw, vertex] (2) at (-1,-1) {};
\node[draw, vertex] (3) at (1,-1) {};
\node[draw, vertex] (4) at (0,-2) {};

\draw (1) -- (2);
\draw (1) -- (3);
\draw (3) -- (4);
\draw (2) -- (4);

\end{tikzpicture}

\end{figure}

\begin{theorem}
Let $G$ have pebbling assignment $(S_G)$. Then $[S_G]$ contains a downward 4-cycle if and only if $G$ has two movable vertices or a $2$-movable vertex with at least four pebbles on it.
\end{theorem}

\begin{proof}
Let $v$ be a $2$-movable vertex of $G$ with at least four pebbles on it. Then since $v$ is $2$-movable there exists two edges $e_1,e_2$ leaving $v$. Since $e_1\neq e_2$ we have that there are two distinct vertices $x,y\in V(G)$ giving the edges $e_1$ and $e_2$ as $(v,x)$ and $(v,y)$, respectively. So we have the sequence of pebbling moves that is from $v$ to $y$ and then $v$ to $x$ as well as the sequence of pebbling moves that is from $v$ to $x$ and then from $v$ to $y$. Since $x\neq y$ and the end state is not dependent on the order of independent moves, this gives us a downward 4-cycle in $[S_G]$. 

Let $v_1,v_2$ be two distinct movable vertices of $G$. Then as before we have edges $(v_1,w_1)$ and $(v_2,w_2)$ and there are at least two pebbles on both $v_1$ and $v_2$. Note that we do not require that $w_1\neq w_2$ only that $v_1\neq v_2$. So we have the sequence of pebbling moves that is from $v_1$ to $w_1$ and then from $v_2$ to $w_2$ as well as the sequence of pebbling moves that is from $v_2$ to $w_2$ and then from $v_1$ to $w_1$. As explained above, this creates a downward 4-cycle in $[S_G]$. 

Conversely, suppose that $[S_G]$ contains a downward 4-cycle. So $[S_G]$ contains the following cycle.

\begin{figure}[h]
\centering

\begin{tikzpicture}[thick, ->]

\node[draw, circle] (1) at (0,0) {A};
\node[draw, circle] (2) at (-1,-1) {B};
\node[draw, circle] (3) at (1,-1) {C};
\node[draw, circle] (4) at (0,-2) {D};

\draw (1) -- (2);
\draw (1) -- (3);
\draw (3) -- (4);
\draw (2) -- (4);

\end{tikzpicture}

\end{figure}

First suppose that the pebbling move from state $A$ to state $B$ and the pebbling move from state $A$ to state $C$ pebble from the same vertex $v\in V(G)$. Since we assume $B\neq C$ as states, we must have that $v$ is a movable vertex of valence at least two. Consider the edges $(v,x)$ and $(v,y)$ in $E(G)$ where the pebbling move from $A$ to $B$ corresponds to pebbling from $v$ to $x$ and the move from $A$ to $C$ corresponds to pebbling from $v$ to $y$. Again since $B\neq C$, we must have that $x\neq y$. Consider the tuple $(n,k,m)_A$ denoting that at state $A$, $n$ is the number of pebbles on vertex $v$, $k$ the number on vertex $x$, and $m$ the number on vertex $y$. So for instance we have the tuples $(n-2,k+1,m)_B$ and $(n-2,k,m+1)_C$. Now since both must end at state $D$, the tuples must be altered by pebbling moves to be equal. If $B$ to $D$ corresponded to the pebbling move from $y$ to $x$ then we have $k+2$ pebbles on vertex $x$ at state $D$ which is not possible to achieve in one pebbling move starting at state $C$. If $B$ to $D$ corresponded to the pebbling move from $x$ to $y$ then we have $(n-2,k-1,m+1)_D$. But starting from $C$, we must have the pebbling move from $x$ to $z$ where $z$ is distinct from $v$ and $y$ in order to achieve the same tuple. But this is not possible, since $z$ was never pebbled to in the $A$ to $B$ to $D$ sequence. A symmetric argument shows that $C$ to $D$ cannot correspond to a pebbling move between $x$ and $y$. So $B$ to $D$ must correspond to the pebbling move from $u$ to $x$ where $u$ is distinct from $v$ and $y$. (Note it must pebble to $x$ because if we pebble to $v$ or $y$ then we would need to change all three coordinates of $(n-2,k,m+1)_C$ to get the same tuple which is not possible in one pebbling move.) So if $u\neq v$ then we have found two movable pebbles and if $u=v$ then since we have moved from $v$ consecutively, there must be at least four pebbles on $v$ (which we already know is $2$-movable). 

If the pebbling move from $A$ to $B$ and the pebbling move from $A$ to $C$ pebble from different vertices then we clearly have two movable vertices. 
\end{proof}

\begin{theorem}
Let $G$ be a graph with a downward 4-cycle and a fully traversable pebbling assignment $(S_G)$. Then $G$ is not isomorphic to $[S_G]$.
\end{theorem}

\begin{proof}
Suppose by way of contradiction that $G\cong [S_G]$. Then since $G$ has contains a downward 4-cycle, $[S_G]$ contains a downward 4-cycle and so by Theorem 2.1, $G$ has two movable vertices or a $2$-movable vertex with at least four pebbles on it. However, by Corollary 1.1, no vertex can have four or more pebbles on it. So $G$ has two movable vertices. Since $G\cong [S_G]$, $G$ is a rooted graph. Let $r$ be the root vertex of $G$. Since $r$ is the root of $G$ and $(S_G)$ is fully traversable, $r$ must be movable. Let $\deg(r)=m$. Now let $R$ be the root vertex of $[S_G]$. We have the $m$ distinct starting pebbling moves from $r$ to each of its neighbors. Then since $G$ has two movable vertices, we have another starting pebbling move from that vertex. Hence, $\deg(R)\geq m+1$. But since $G\cong [S_G]$ we must have that $\deg(r)=\deg(R)$. Hence, we have a contradiction.
\end{proof}

\begin{corollary}
Let $G$ be the downward 4-cycle. The following assignments are the only assignments on $G$ where $G\cong [S_G]$ (up to graph symmetry). 

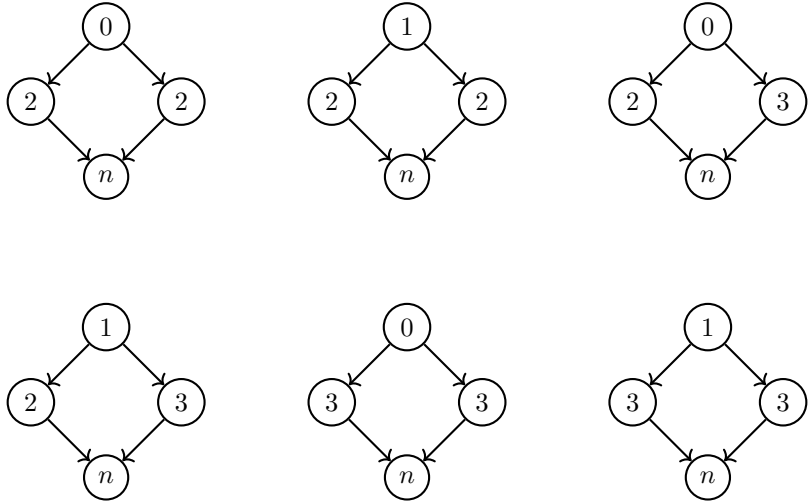
\begin{figure}[h]
\centering 

\begin{tikzpicture}[thick, ->]

\node[draw, circle] (1) at (0,0) {0};
\node[draw, circle] (2) at (-1,-1) {2};
\node[draw, circle] (3) at (1,-1) {2};
\node[draw, circle] (4) at (0,-2) {$n$};

\node[draw, circle] (5) at (4,0) {1};
\node[draw, circle] (6) at (3,-1) {2};
\node[draw, circle] (7) at (5,-1) {2};
\node[draw, circle] (8) at (4,-2) {$n$};

\node[draw, circle] (9) at (8,0) {0};
\node[draw, circle] (10) at (7,-1) {2};
\node[draw, circle] (11) at (9,-1) {3};
\node[draw, circle] (12) at (8,-2) {$n$};


\node[draw, circle] (13) at (0,-4) {1};
\node[draw, circle] (14) at (-1,-5) {2};
\node[draw, circle] (15) at (1,-5) {3};
\node[draw, circle] (16) at (0,-6) {$n$};

\node[draw, circle] (17) at (4,-4) {0};
\node[draw, circle] (18) at (3,-5) {3};
\node[draw, circle] (19) at (5,-5) {3};
\node[draw, circle] (20) at (4,-6) {$n$};

\node[draw, circle] (21) at (8,-4) {1};
\node[draw, circle] (22) at (7,-5) {3};
\node[draw, circle] (23) at (9,-5) {3};
\node[draw, circle] (24) at (8,-6) {$n$};

\draw (1) -- (2);
\draw (1) -- (3);
\draw (3) -- (4);
\draw (2) -- (4);

\draw (5) -- (6);
\draw (5) -- (7);
\draw (7) -- (8);
\draw (6) -- (8);

\draw (9) -- (10);
\draw (9) -- (11);
\draw (11) -- (12);
\draw (10) -- (12);

\draw (13) -- (14);
\draw (13) -- (15);
\draw (15) -- (16);
\draw (14) -- (16);

\draw (17) -- (18);
\draw (17) -- (19);
\draw (19) -- (20);
\draw (18) -- (20);

\draw (21) -- (22);
\draw (21) -- (23);
\draw (23) -- (24);
\draw (22) -- (24);

\end{tikzpicture}

\end{figure}

\end{corollary}

\begin{proof}
By drawing out the possible pebbling moves one can easily check that these all have the property that $G\cong [S_G]$. Now it suffices to prove that these are the only possible assignments. By Theorem 2.1, $G$ must have either two movable vertices or a $2$-movable vertex with at least four pebbles on it. Note that $G$ only has one vertex with a valence of at least two and that is the root. So consider $G$ with the assignment $(R_G)$ below.

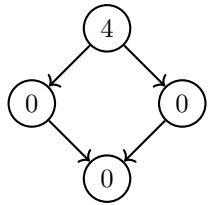
\begin{figure}[h]
\centering

\begin{tikzpicture}[thick, ->]

\node[draw, circle] (1) at (0,0) {4};
\node[draw, circle] (2) at (-1,-1) {0};
\node[draw, circle] (3) at (1,-1) {0};
\node[draw, circle] (4) at (0,-2) {0};

\draw (1) -- (2);
\draw (1) -- (3);
\draw (3) -- (4);
\draw (2) -- (4);

\end{tikzpicture}

\end{figure}

This assignment has the assignment graph $[R_G]$ as follows. 

\begin{figure}[h]
\centering

\begin{tikzpicture}[thick, ->]

\tikzstyle{vertex}=[circle,fill=black,inner sep=2pt]

\node[draw, vertex] (1) at (0,0) {};
\node[draw, vertex] (2) at (-1,-1) {};
\node[draw, vertex] (3) at (1,-1) {};
\node[draw, vertex] (4) at (0,-2) {};
\node[draw, vertex] (5) at (-1,-2) {};
\node[draw, vertex] (6) at (1,-2) {};
\node[draw, vertex] (7) at (0,-3) {};

\draw (1) -- (2);
\draw (1) -- (3);
\draw (3) -- (4);
\draw (2) -- (4);
\draw (2) -- (5);
\draw (3) -- (6);
\draw (5) -- (7);
\draw (6) -- (7);

\end{tikzpicture}

\end{figure}

Then since any assignment $(S_G)$ on $G$ with a $2$-movable vertex with at least four pebbles on it has the $(R_G)$ as a subassignment, it follows from the standard case that $[R_G]$ will be isomorphic to a subgraph of $[S_G]$ and thus, $[S_G]$ cannot be isomorphic to $G$. Hence, $G$ has two movable vertices. By Theorem 2.2, $G$ cannot be fully traversable. Hence, $G$ must have exactly two movable vertices since the bottom vertex is never movable and thus three movable vertices makes the graph assignment fully traversable which is a contradiction.

Consider $G$ with the assignment $(R_G)$ below. 

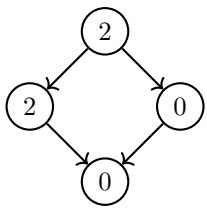
\begin{figure}[h]
\centering

\begin{tikzpicture}[thick, ->]

\node[draw, circle] (1) at (0,0) {2};
\node[draw, circle] (2) at (-1,-1) {2};
\node[draw, circle] (3) at (1,-1) {0};
\node[draw, circle] (4) at (0,-2) {0};

\draw (1) -- (2);
\draw (1) -- (3);
\draw (3) -- (4);
\draw (2) -- (4);

\end{tikzpicture}

\end{figure}

Its assignment graph's node vertex has valence three since there are three distinct pebbling moves we can make. Then again since for any assignment $(S_G)$ with these two vertices as the chosen movable vertices, $(R_G)$ is a subassignment and thus $[R_G]$ is isomorphic to a subgraph of $[S_G]$. But since the root of $[R_G]$ has valence $3$, $[S_G]$ is not isomorphic to $G$. Thus, the root vertex cannot be movable. So the two movable vertices must be the two side vertices. Finally consider $G$ with the assignment $(R_G)$ below. 

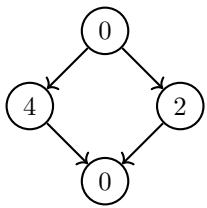
\begin{figure}[h]
\centering

\begin{tikzpicture}[thick, ->]

\node[draw, circle] (1) at (0,0) {0};
\node[draw, circle] (2) at (-1,-1) {4};
\node[draw, circle] (3) at (1,-1) {2};
\node[draw, circle] (4) at (0,-2) {0};

\draw (1) -- (2);
\draw (1) -- (3);
\draw (3) -- (4);
\draw (2) -- (4);

\end{tikzpicture}

\end{figure}

This assignment has the assignment graph $[R_G]$ as follows. 

\begin{figure}[h]
\centering

\begin{tikzpicture}[thick, ->]

\tikzstyle{vertex}=[circle,fill=black,inner sep=2pt]

\node[draw, vertex] (1) at (0,0) {};
\node[draw, vertex] (2) at (-1,-1) {};
\node[draw, vertex] (3) at (1,-1) {};
\node[draw, vertex] (4) at (0,-2) {};
\node[draw, vertex] (5) at (1,1) {};
\node[draw, vertex] (6) at (2,0) {};

\draw (1) -- (2);
\draw (2) -- (4);
\draw (1) -- (3);
\draw (3) -- (4);
\draw (5) -- (1);
\draw (5) -- (6);
\draw (6) -- (3);

\end{tikzpicture}

\end{figure}

So by the same argument we have been using $[S_G]$ is not isomorphic to $G$. Thus, we have shown that $(S_G)$ must have both the side vertices as movable vertices, the root not a movable vertex, and the side vertices must have less than four pebbles each on them. Hence, we are done.
\end{proof}

\section{Larger Downward Cycles}

\begin{theorem}
Let $G$ be a downward $k$-cycle for $k>4$. Then for any pebble assignment $(S_G)$, $G$ is not isomorphic to $[S_G]$.
\end{theorem}

\begin{proof}
Suppose by way of contradiction that there exists a pebbling assignment $(S_G)$ such that $G\cong [S_G]$. Clearly, $k$ must be an even integer such that $k\geq 4$ since $[S_G]$ is a simple bipartite graph. Since $G$ must have at least one movable vertex because $[S_G]$ is non-trivial, it has exactly one since otherwise $[S_G]$ would contain a downward $4$-cycle and hence $G$ would as well. If $(S_G)$ is fully traversable then the root node of $G$ must be the movable vertex. Also, since $(S_G)$ is fully traversable with only one movable vertex, the sides of the downward cycle must all have exactly one pebble on them. But then $[S_G]$ is the root vertex that stems into two identical, downward paths. Hence, $[S_G]$ is not a downward cycle and thus we have a contradiction. If $(S_G)$ is not fully traversable then by Corollary 1.2 some edge of $G$ is traversed more than once. However, this implies that some vertex of $G$ has at least four pebbles placed on it. If the root vertex had at least four pebbles on it then since it is valence two, by Theorem 2.1, $[S_G]$ would have a downward 4-cycle but this is a contradiction since $[S_G]\cong G$. Hence, the vertex with four pebbles placed on it must be on one of the sides of the cycle. But since this vertex must also be the only movable vertex, $[S_G]$ has no cycle and thus we reach a contradiction. 
\end{proof}

\section{Fully Traversable Cycles}

We now want to address any cycle, not just downward ones. Note that in an assignment graph the orientation of the edges will never give us a directed cycle so when we refer to cycles we are referring to the underlying undirected graph in both the assignment graph and the original graph. 

\begin{theorem}
Let $G$ be a graph with fully traversable pebbling assignment $(S_G)$. If $G$ contains a cycle then $G$ is not isomorphic to $[S_G]$.
\end{theorem}

\begin{proof}
Suppose by way of contradiction $G\cong [S_G]$. So $G$ must be connected and since it is fully traversable, every vertex must be reachable by some sequence of pebbling moves. By Theorems 2.1 and 2.2, $G$ has exactly one movable vertex. Then by Corollary 1.1, this movable vertex has at most three pebbles on it. Let $G$ have $m$ vertices with zero valence (note that $m>0$ since every sequence of pebbling moves is finite and $G\cong [S_G]$). Let $v_1,...,v_m$ be these $m$ vertices. Since $(S_G)$ is fully reachable, we have a sequence of pebbling moves from the movable vertex to each $v_i$ for all $1\leq i\leq m$. Since these sequences end with no movable vertices, and they each can only affect the number of pebbles on one of the $v_i$, they correspond to distinct end states. Thus, we have found $m$ vertices of $[S_G]$ with zero valence. However, since $(S_G)$ is fully reachable and $G$ contains a cycle, we can find a two distinct sequences of pebbling moves from the movable vertex to one of the $v_i$. Hence, these correspond to two different end states since each sequence traverses different edges. But since both sequences end at $v_i$, these two end states are distinct from the $m-1$ end states we found corresponding the the $v_j$ for $j\neq i$. Hence, we have found $m+1$ distinct end states and thus $[S_G]$ has $m+1$ vertices with valence zero. Since $G\cong [S_G]$ this is a contradiction.
\end{proof}

\section{Trees}

\begin{theorem}
Let $T$ be a downward directed, rooted tree. Let $(S_T)$ be the pebbling assignment on $T$ with two or three pebbles on the root, one pebble on each vertex with non-zero valence that is not the root vertex, and any number of pebbles on the valence zero vertices. Then $T\cong [S_T]$. 
\end{theorem}

\begin{proof}
Let $T$ be a downward directed, rooted tree with pebbling assignment $(S_T)$ as described. For every vertex $v\in V(T)$, there is a unique directed path from the root vertex $r$ of $T$ to $v$. Since there is two or three pebbles on $r$ and one pebble on every non-zero valence vertex of $T$, $(S_T)$ is fully traversable and so we can add one pebble to vertex $v$ starting at $r$. In particular, we must start at $r$ since it is the only movable vertex. Let $A_v$ denote that state after the seqeuence of pebbling moves adding one pebble to $v$. We know that $A_v$ is uniquely determined by $v$ since it is uniquely determined by the directed path from $r$ to $v$ which is unique. Now consider the map $\psi:T\to [S_T]:v\mapsto A_v$. If $A_{v_1}=A_{v_2}$ then the path from $r$ to $v_1$ and the path from $r$ to $v_2$ are the same. Hence, $v_1=v_2$ and so $\psi$ is injective. Given any pebbling state, it must start at $r$ and at any time there is only one movable vertex and so the sequence of pebbling states follows as directed path. Hence, $\psi$ is surjective as well. If $(v_1,v_2)$ is a directed edge in $T$ then we can go from state $A_{v_1}$ to state $A_{v_2}$ by pebbling from $r$ to $v_1$ to $v_2$. So $(A_{v_1},A_{v_2})\in E([S_T])$. If $(A_{v_1},A_{v_2})$ is a directed edge in $[S_T]$ then we can get from state $A_{v_1}$ to state $A_{v_2}$ in one pebbling move. Since $A_{v_1}$ ends at vertex $v_1$ and $A_{v_2}$ ends at vertex $v_2$, we must be able to pebble from $v_1$ to $v_2$ in one move. Hence, $(v_1,v_2)\in E(T)$. So $\psi$ is an isomorphism. 
\end{proof}

\section{Classification of Fully Traversable Graphs}

We now want to precisely classify all graph $G$ with fully traversable pebbling assignments $(S_G)$ such that $G\cong [S_G]$. We have shown that $G$ cannot contain a cycle. Since $[S_G]$ is always a rooted, downward directed, simple graph, $G$ must be a downward directed, rooted tree. Furthermore, if $G$ has two movable vertices then $[S_G]$ contains a downward 4-cycle which means $G$ cannot be isomorphic to $[S_G]$. Also, since $(S_G)$ is fully traversable, the root node must be the movable vertex and it can only have two or three pebbles on it. Furthermore, since $(S_G)$ is fully traversable every non-zero valence vertex must have exactly one pebble on it. So by the previous section we have shown that a graph $G$ with fully traversable pebbling assignment $(S_G)$ has $G\cong [S_G]$ if and only if $G=T$ and $(S_G)=(S_T)$ as described in the previous section.

\section{Not Fully Traversable Graphs}

\begin{definition}
Let $P_m$ be a oriented path. Then a \textbf{simple} pebbling assignment on $P_m$ is an assignment with two or three pebbles on the source vertex, one pebble on each non-sink vertex, and any number of pebbles on the sink vertex. 
\end{definition}

Consider oriented paths $P_{n_1},...,P_{n_r}$ each with a simple pebbling assignment. Let us label each path such that we label the source of path $P_{n_i}$ as $a_{i,1}$, the next vertex, $a_{i,2}$, and continue until we label the sink $a_{i,n_i}$. Now consider the graph $\square_{i=1}^r P_{n_i}$. Each vertex can be labeled by coordinates relating to the Cartesian product. Consider the path
\begin{align*}
(a_{1,1},a_{2,n_2},...,a_{r,n_r}),(a_{1,2},a_{2,n_2},...,a_{r,n_r}),...,(a_{1,n_1},a_{2,n_2},...,a_{r,n_r}).
\end{align*}
Clearly, this path is isomorphic to $P_{n_1}$. Hence, we pebble it with $(S_{P_{n_1}})$. Similarly we pebble 
\begin{align*}
a_{1,n_1},a_{2,1},...,a_{r,n_r}),(a_{1,n_1},a_{2,2},...,a_{r,n_r}),...,(a_{1,n_1},a_{2,n_2},...,a_{r,n_r})
\end{align*} 
 with $(S_{P_{n_2}})$ and so on. Note that we may have a discrepancy of what to pebble $(a_{1,n_1},a_{2,n_2},...,a_{r,n_r})$, but this does not matter since this is a sink vertex. Hence, we can put any number of pebbles here as we did with each $P_{n_i}$. Then put zero or one pebbles on each remaining vertex. We will call this a simple pebbling on $\square_{i=1}^r P_{n_i}$. 

\begin{theorem}
Let $G=\square_{i=1}^r P_{n_i}$ have a simple pebbling on it. Then $G\cong [S_G]$. 
\end{theorem}

\begin{proof}
By our construction of $G$, pebbling moves on different paths are independent of each other. Hence, every vertex of $[S_G]$ can be labelled with a sequence of pebbling moves (where order does not matter). Let $b_i$ refer to the unique pebbling move on the path isomorphic to $P_{n_i}$. Furthermore, we denote $b_i^2$ as two moves on that path and so on. Hence, we can label every vertex of $[S_G]$ where the start state is the empty label. Now consider $f:V([S_G])\to V(G)$ where $f:b_1^{k_1}\cdots b_r^{k_r}\mapsto (a_{1,k_1},...,a_{r,k_r})$. It's easy to see that this is a bijection. Let us check now that it preserves edges and their orientation. An edge of $G$ has all coordinates fixed except one which goes from $a_{i,j}$ to $a_{i,j+1}$ (we call this positively oriented). Similarly, an edge of $[S_G]$ increases the exponent of one $b_i$ by one (again this is positively oriented). Hence, $f$ preserves edges and positive orientation and is thus an isomorphism of directed graphs. 
\end{proof}

\begin{definition}
Let $P_n$ be an oriented path. Then an \textbf{almost simple} pebbling assignment of $P_n$ is a simple pebbling or one of the following:
\begin{enumerate}[1.]
\item $P_n$ with $k$ pebbles on the vertex adjacent to the sink, $m$ pebbles on the sink, and zero or one pebbles on all other vertices;
\item $P_n$ with four or five pebbles on some vertex, zero pebbles on the following vertex, one or zero pebbles on all other non-sink vertices, and $m$ pebbles on the sink vertex.
\end{enumerate}
\end{definition}

\begin{lemma}
Let $P_n$ be an oriented path with pebbling assignment as in definition 7.2.1 such that $n=\lfloor k/2\rfloor+1$. Then $P_n\cong [S_{P_n}]$. 
\end{lemma}

\begin{proof}
At every stage of the pebbling process, we have at most one move because of the orientation of the graph and the only movable vertex leads directly into the sink which can never be movable. Hence, we have one sequence of $\lfloor k/2\rfloor+1$ pebbling moves (where the $+1$ comes from included the start state). So $[S_{P_n}]$ is an oriented path with $\lfloor k/2\rfloor +1$ vertices. Hence, $P_n\cong [S_{P_n}]$.  
\end{proof}

\begin{lemma}
Let $P_n$ be an oriented path with pebbling assignment as in example 7.2.2 such that exactly $n-2$ of the edges are traversable. Then $P_n\cong [S_{P_n}]$. 
\end{lemma}

\begin{proof}
We know that the edge originating from the vertex with four or five pebbles is traversed twice. Since $n-2$ of the edges are traversed in the unique pebbling sequence (since there is always at most one movable vertex) we have a total of $n-1$ pebbling moves and hence $[S_{P_n}]$ is an oriented path on $n$ vertices. Thus, $P_n\cong [S_{P_n}]$. 
\end{proof}

Now consider oriented paths $P_{n_1},...,P_{n_r}$ such that they all have an almost simple pebbling assignment on them where $P_{n_i}\cong [S_{P_{n_i}}]$. Then using the same construction as before we have the graph $G=\square_{i=1}^r P_{n_i}$ and we can give it an almost simple pebbling assignment following the same procedure. Hence we can identity each $P_{n_i}$ and its pebbling assignment within $G$. (Again recall it does not matter how we pebble the sink vertex that all the paths converge to.) Thus we have the following corollary. 

\begin{corollary}
Let $G=\square_{i=1}^r P_{n_i}$ with an almost simple pebbling assignment. Then $G\cong [S_G]$. 
\end{corollary}

\begin{theorem}
Let $K_{n,m}:=\overline{K_n}\vee\overline{K_m}$. We orient $K_{n,m}$ by having the edges leaving $\overline{K_n}$ and towards $\overline{K_m}$. Then there exists a graph $G$ and pebbling assignment $(S_G)$ such that $G$ contains $K_{n,m}$ and $G\cong [S_G]$.  
\end{theorem}

\begin{proof}
Let us pebble $K_{n,m}$ with two or three pebbles on the $n$ vertices of $\overline{K_n}$ and any number of pebbles on the $m$ vertices of $\overline{K_m}$. Let $G=[S_{K_{n,m}}]$. Currently, $G$ is naturally oriented and not pebbled. Every pebbling move on $K_{n,m}$ is uniquely described by which edge is traversed (i.e. by the pair $(u,v)$ for $u\in \overline{K_n}$ and $v\in \overline{K_m}$). Hence if we label $\overline{K_n}$ as $\brac{a_1,...,a_n}$ and $\overline{K_m}$ as $\brac{b_1,...,b_m}$, then the edges of $K_{n,m}$ are labeled $\brac{(a_i,b_j)}$. So we can label the vertices of $G$ as sequences of these $(a_i,b_j)$ which will be unordered since the order of pebbling moves does not matter. Note that the sink nodes of $G$ will be labeled as a sequence in which the moves $(a_i,b_j)$ occur for every $a_i$ paired with some $b_j$ (i.e. exhausting all pebbled nodes of $K_{n,m}$). Consider the map $f:V(K_{n,m})\to V(G)$ where $f: a_i\mapsto (a_1,b_1)(a_2,b_1)\cdots (a_{i-1},b_1)(a_{i+1},b_1)\cdots (a_n,b_1)$
\end{proof}

\pagebreak

\section{Undirected Isomorphism}

In this section we consider oriented graphs $G$ with pebbling assignment $(S_G)$ such that $G\cong [S_G]$ as undirected graphs. 

\begin{theorem}
Let $G$ be an oriented graph with pebbling assignment $(S_G)$. If $G$ is an undirected induced subgraph of $[S_G]$ then there exists an oriented graph $H$ with pebbling assignment $(S_H)$ such that $G$ is an oriented subgraph of $H$ and $(S_G)$ is a restricted assignment of $(S_H)$ and $H\cong [S_H]$. 
\end{theorem}

\begin{proof}
Let $H=[S_G]$ as an undirected graph. Then $G$ is a subgraph of $H$. Let us copy the orientation and pebble assignment of $G$ onto the copy of $G$ within $H$. Since $G$ is an induced subgraph of $H$, every edge of $H$ that is not in $G$ is incident to at most one vertex of $G$. Hence, for any of these edges we orient them towards $G$. The rest of the edges may be oriented in whatever fashion. Let us pebble the rest of $H$ with zero pebbles. Then $H$ is not fully-traversable and in particular, $[S_H]=[S_G]$. Hence, $H\cong[S_H]$. 
\end{proof}

Note that we can have a similar theorem for when the isomorphic preserved orientation but we would have to require that not only is $G$ downward oriented but also when adding all edges incident to $G$ we may orient them downward. 

\section*{Acknowledgments}

This material is based upon work that was supported by the National Science Foundation under Grant DMS-1852378. Any opinions, findings, and conclusions or recommendations expressed in this material are those of the authors and do not necessarily reflect the views of the National Science Foundation.

We would like to thank our research advisors, Eugene Fiorini, Nathan Shank, and Andrew Woldar, for their guidance on this project and for the ideas that inspired all of this work. Finally, we want to thank Moravian College for hosting the REU this summer.

\pagebreak 

\section*{References}

\begin{enumerate}[\text{[}1\text{]}]

\item M. S. Anilkumar and S. Sreedevi,  ``A comprehensive review on graph pebbling and rubbling," \textit{Jour. of Physics: Conference Series}, \textbf{1531} (2019).

\item D. H. Bailey and J. M. Borwein, ``Experimental Mathematics: Examples, Methods, and Implications, \textit{Notices of the AMS}, \textbf{52} (2005), 502-514.

\item E. R. Berlekamp, J. H. Conway, and R. K. Guy, ``Winning Ways for Your Mathematical Plays, Vol. 1," \textit{CRC Press}, (2001).

\item J. Blocki and S. Zhou, ``On the computational complexity of minimal cumulative cost graph pebbling," \textit{Financial Cryptography and Data Security, Lecture Notes in Computer Science, Springer, Berlin, Heidelberg}, \textbf{10957} (2018).

\item A. Czygrinow and G. Hurlbert, ``Girth, pebbling, and grid thresholds," \textit{SIAM J. Discrete Math}. \textbf{20} (2006), 1-10.

\item E. Fiorini, G. Johnston, M. Lind, A. Woldar, and W. H. T. Wong, ``Cycles and girth in pebble assignment graphs," \textit{Graphs and Combinatorics}, submitted.

\item E. Fiorini, M. Lind, A. Woldar, and W. H. T. Wong, ``Characterizing winning positions in the impartial two-player pebbling game on complete graphs," \textit{J. of Int. Seq}. \textbf{24(6)}, (2021).

\item E. Fiorini, V. de Silva, and C. J. Verbeck, Jr., ``Symmetric class-0 subgraphs of complete graphs, \textit{DIMACS Technical Reports}, (2011).

\item G. Isaak and M. Prudente, ``Two-player pebbling on diameter 2 graphs," \textit{Int. J. Game Theory}, (2021).

\item M. Lind, E. Fiorini, and A. Woldar, ``On properties of pebble assignment graphs," \textit{Graphs and Combinatorics}, submitted.

\item G. Hurlbert, ``General graph pebbling," \textit{Discrete Applied Mathematics} \textbf{161} (2013), 1221-1231.

\item G. Hurlbert, ``Recent progress in graph pebbling," \textit{arXiv:math/0509339v1 [math.CO]} (2008).

\item G. Hurlbert, ``A survey of graph pebbling," \textit{arXiv:math/0406024v1 [math.CO]} (2004).

\end{enumerate}

\end{document}